\documentclass[reqno, 12pt]{amsart} 
\usepackage{array} 
\usepackage{amsmath}
\usepackage{amsfonts}
\usepackage{amssymb}
\usepackage{mathrsfs}
\usepackage{enumerate}
\usepackage{amsthm}
\usepackage{verbatim}
\usepackage{amsmath, amscd}
\usepackage[usenames,dvipsnames]{color}
\usepackage{enumitem}
\usepackage{bm}
\usepackage{xy}
\xyoption{all}
\usepackage{color}
\usepackage{marvosym}
\usepackage{amsbsy}
\usepackage{hyperref}
\usepackage{tikz, tikz-cd}
\usepackage{marginnote}
\usetikzlibrary{matrix,arrows,backgrounds}%, decorations, pathmorphing, positioning, fit}
\usepackage[backgroundcolor=white]{todonotes}
\usepackage{hyperref}

\newtheorem{theorem}{Theorem}[section]

\newtheorem{lemma}[theorem]{Lemma}
\newtheorem{proposition}[theorem]{Proposition}
\newtheorem{corollary}[theorem]{Corollary}

\theoremstyle{definition}
\newtheorem{definition}[theorem]{Definition}
\newtheorem{remark}[theorem]{Remark}

\newcommand{\op}[1]{\operatorname{#1}}

\newcommand{\newterm}{\textsf}

\newcommand{\dbcoh}[1]{\operatorname{D}^{\operatorname{b}}(\operatorname{coh }#1)}

\newcommand{\dabs}{\mathrm{D}}
\newcommand{\gm}{\mathbb{G}_m}

\newcommand{\I}{\mathcal I}
\newcommand{\J}{\mathcal J}

\newcommand{\Hom}{\text{Hom}}

\def\ker{\op{\mbox{Ker}}}

\def\Z{\op{\mathbb{Z}}}
\def\C{\op{\mathbb{C}}}

\def\Q{\op{\mathbb{Q}}}

\def\O{\op{\mathcal{O}}}
\def\A{\op{\mathbb{A}}}

\def\Gm{\op{\mathbb{G}_m}}
\def\I{\op{\mathcal{I}}}
\def\J{\op{\mathcal{J}}}

\title[Exceptional Collections for Mirrors of Invertible Polynomials]{Exceptional Collections for Mirrors of Invertible Polynomials}

\author[Favero]{David Favero}
\address{
	\begin{tabular}{l}
		David Favero \\
		\hspace{.1in} University of Minnesota \\
		\hspace{.1in} 454 Vincent Hall, 206 Church Street SE, Minneapolis, MN 55455, USA  \smallskip \\
        \hspace{.1in} Korea Institute for Advanced Study \\ \hspace{.1in} 85 Hoegiro, Dongdaemun-gu, Seoul, Republic of Korea 02455\\
		\hspace{.1in} Email: {\bf favero@umn.edu} \\
	\end{tabular}
}

\author[Kaplan]{Daniel Kaplan}
\address{
  \begin{tabular}{l}
   Daniel Kaplan \\
   \hspace{.1in} Universiteit Hasselt \\
   \hspace{.1in} Universitaire Campus, 3590 Diepenbeek, Belgium \\
   \hspace{.1in} Email: {\bf daniel.kaplan@uhasselt.be} \\
  \end{tabular}
}

\author[Kelly]{Tyler L. Kelly}
\address{
  \begin{tabular}{l}
   Tyler L. Kelly \\
   \hspace{.1in} University of Birmingham \\
   \hspace{.1in} School of Mathematics, Edgbaston, Birmingham B15 2TT,  United Kingdom \\
   \hspace{.1in} Email: {\bf t.kelly.1@bham.ac.uk} \\
  \end{tabular}
}

\numberwithin{equation}{section}
\begin{document}

%\linenumbers

\begin{abstract}
We prove the existence of a full exceptional collection for the derived category of equivariant matrix factorizations of an invertible polynomial with its maximal symmetry group.  This proves a conjecture of Hirano--Ouchi.  
In the Gorenstein case, we also prove a stronger version of this conjecture due to Takahashi.  Namely, that the full exceptional collection is strong.
\end{abstract}

\maketitle
\setcounter{tocdepth}{1}
%\tableofcontents
 \section{Introduction}

Let $\C$ be an algebraically closed field of characteristic zero. We say that a polynomial $w\in \C[x_1,\ldots, x_n]$ is \newterm{invertible} if it is of the form
\[
w = \sum_{i=1}^n \prod_{j=1}^n x_j^{a_{ij}}
\]
where $A=(a_{ij})_{i,j=1}^n$ is a non-negative integer-valued matrix satisfying:
\begin{enumerate}
\item $A$ is invertible over $\Q$;
\item $w$ is quasihomogeneous, i.e., there exists positive integers $q_j$ such that $d:= \sum_{j=1}^n q_j a_{ij}$ is constant for all $i$; and
\item $w$ is quasi-smooth, i.e., $w:\mathbb{A}^n \to \mathbb{A}^1$ has exactly one critical point (at the origin).
\end{enumerate}

Let $\Gm$ be the multiplicative torus. We may consider the following group of symmetries:
\begin{equation} \label{eq: GammaW}
\Gamma_w := \{ (t_1, \ldots, t_{n+1}) \in \mathbb{G}_m^{n+1} \ | \ w(t_1x_1, \ldots, t_nx_n) = t_{n+1}w(x_1,\ldots, x_n) \}.
\end{equation}
The group $\Gamma_w$ acts on $\mathbb{A}^n$ by projecting onto the first $n$ coordinates and then acting diagonally. The Landau-Ginzburg model $(\mathbb{A}^n, \Gamma_w, w)$ is a proposed mirror of the transposed invertible polynomial
\[
w^T = \sum_{i=1}^n \prod_{j=1}^n x_j^{a_{ji}}.
\]

Kontsevich's Homological Mirror Symmetry Conjecture predicts that the Fukaya-Seidel category of $w^T$ \cite{Sei08} is equivalent to the (gauged) matrix factorization category  $\dabs[\mathbb{A}^n, \Gamma_w, w]$ \cite{Pos11, BFK14a}. A few cases of this equivalence have been proven. When $w$ is a Fermat polynomial, meaning $w = \sum_{i=1}^n x_i^{r_i}$, this equivalence is proven by Futaki and Ueda \cite{FU09, FU11}. When $n=2$, the conjecture has been proven by Habermann and Smith \cite{HS19}. The approach of Futaki-Ueda and Habermann-Smith involves finding matching tilting objects for $\dabs{[\mathbb{A}^n, \Gamma_w, w]}$ and $\mathcal{F}(w^T)$. 
This makes the existence of a tilting object on $\dabs{[\mathbb{A}^n, \Gamma_w, w]}$ for arbitrary $n$ and $w$ desirable.

 In fact, the existence of such a tilting object was first conjectured informally by Takahashi \cite{TakashiSlides} during his presentation at the University of Miami in 2009.  Therein, he demonstrated that the results of \cite{KST1, KST2} imply new cases of homological mirror symmetry when $n=3$. In the literature, this existence has been conjectured by Ebeling and Takahashi in three-variables \cite{ET11}, and by Lekili and Ueda in general \cite[Conjecture 1.3]{LU18}. The conjecture was weakened more recently by Hirano and Ouchi who ask for the existence of a full exceptional collection (which is not necessarily strong) \cite[Conjecture 1.4]{HO} (see \S\ref{subsec: exceptional} for definitions).
 
 Due to the Kreuzer-Skarke classification of invertible polynomials \cite{KS92} (in fact this classification can be traced further back to Orlik-Wagreich \cite{OW71}), we know that any invertible polynomial, up to permutation of variables, can be written as a Thom-Sebastiani sum of three types of polynomials:

\begin{enumerate}
\item Fermat type: $w = x^r$,
\item Chain type: $w = x_1^{a_1}x_2 + x_2^{a_2} x_3 + \ldots + x_{n-1}^{a_{n-1}}x_n + x_n^{a_n}$, and 
\item Loop type: $w = x_1^{a_1}x_2 + x_2^{a_2} x_3 + \ldots + x_{n-1}^{a_{n-1}}x_n + x_n^{a_n}x_1$.
\end{enumerate}
By Corollary 2.40 of \cite{BFK14b}, the conjectures above on the existence of a full exceptional collection or a tilting object reduce to studying indecomposable invertible polynomials that are of any one given type. 

There is a long history of partial results for various cases using the Kreuzer-Skarke classification. The case of invertible polynomials consisting only of Fermat monomials was established by Takahashi in \cite{Tak1} and the case of certain invertible polynomials in three variables was proven by Kajiura, Saito and Takahashi \cite{KST1, KST2}. In the chain case, Hirano and Ouchi proved the existence of a full exceptional collection  \cite[Corollary 1.6]{HO} and Aramaki and Takahashi were able to then prove the existence of a Lefschetz decomposition \cite{AT19}.\footnote{In the time since the first version of this paper, Hirano and Ouchi extended their result to provide a strong exceptional collection for chain type polynomials.} Moreover, Kravets has proven when $n\le 3$ that $\dabs{[\mathbb{A}^n, \Gamma_w, w]}$ has a full strong exceptional collection \cite{Kra19}.
 
In the present paper, we use variation of GIT techniques for derived categories  \cite{HL15,BFK12} in order to construct a full exceptional collection in all three cases uniformly. 
\begin{theorem}\label{proof of conjecture}
Conjecture 1.4 of \cite{HO} is true: for any invertible polynomial $w$, the singularity category $\dabs{[\mathbb{A}^n, \Gamma_w, w]}$ has a full exceptional collection whose length is equal to the Milnor number of $w^T$. Furthermore, if the dual polynomial $w^T$ has weights $r_i$ and degree $d^T$ such that  $r_i$ divides $d^T$ for all $i$, then Conjecture 1.3 of \cite{LU18} holds: the singularity category $\dabs{[\mathbb{A}^n, \Gamma_w, w]}$ has a tilting object.
\end{theorem}

The techniques of the proof provide a recursive way to relate loop polynomials to chain polynomials and chain polynomials to Thom-Sebastiani sums of Fermat monomials and smaller chain polynomials. This recursive technique has been mirrored in the A-model by Polishchuk and Varolgunes to give very strong evidence of homological mirror symmetry for chain polynomials (up to some formal foundations currently missing on the generation of a certain exceptional collection of thimbles for Fukaya-Seidel categories of tame Landau-Ginzburg models) \cite{PV21}.

\begin{remark}
The divisibility condition in the theorem is equivalent to requiring that the coarse moduli space of $[\mathbb{A}^n / \Gamma_{w^T}]$ is Gorenstein.  In the chain case, this means $a_i$ divides $a_{i+1}$ for $1 \leq i < n$.  In the loop case, this means $a_1 = ... = a_n$.  For Thom-Sebastiani sums, it means the above on each summand.
\end{remark}

\begin{remark}
Theorem~\ref{proof of conjecture} can be interpreted as evidence for a Landau-Ginzburg version of Dubrovin's conjecture \cite{Dub98} as the Frobenius manifold associated to the LG model $(\mathbb{A}^n, \Gamma_w, w)$ is (generically) semi-simple.
\end{remark}

\begin{remark}
By Orlov's Theorem \cite[Theorem 3.11]{Orl09}, one can transform our exceptional collection into a collection of geometric objects living in the maximally-graded derived category of the corresponding hypersurface.  In \cite{FKK20}, it is shown that this collection cannot always consist of line bundles (for the example therein our collection has 32 objects and a bound of 24 line bundles is obtained).
\end{remark}

\begin{remark}
The proof could be made into an algorithm to produce the exceptional collection whose existence is stated in Theorem~\ref{proof of conjecture}.  This is simplest in the Gorenstein case, where the tilting object can be described.
\end{remark}

\subsection{Acknowledgments}
This project was initiated while the three authors visited the Fields Institute during the Thematic Program on the Homological Algebra of Mirror Symmetry.  We would like to thank the Fields Institute for providing this opportunity, their hospitality, excellent atmosphere, and superb afternoon tea.

The original version of this paper mistakenly claimed to provide a counterexample to Conjecture 1.3 of version 2 of the preprint \cite{LU18v2}.  We humbly thank Kazushi Ueda for immediately recognizing this issue and informing us of our blunder. This subject has since been delegated to the paper \cite{FKK20} where a true counterexample is given. We are also grateful to Ailsa Keating, Yanki Lekili, and Atsushi Takahashi for their conversations, comments and suggestions.  We lastly thank the referee for their useful comments.

The first-named author is grateful to the Natural Sciences and Engineering Research Council of Canada for support provided by a Canada Research Chair and Discovery Grant.
The second-named author is thankful for support by the Engineering and Physical Sciences Research Council (EPSRC) under Grant EP/S03062X/1. 
The third-named author acknowledges that this paper is based upon work supported by the EPSRC under Grant EP/N004922/2 and EP/S03062X/1, along with the Birmingham International Engagement Fund.
 \vspace{2.5mm}

\section{Background}

\subsection{The maximal symmetry group of a polynomial}

Let 
\[
W = \sum_{i=1}^k \prod_{j=1}^n x_j^{a_{ij}}
\]
 be a polynomial in $n$ variables with $k$ monomials. Viewing the $A_W^T = (a_{ji})$ as an integer valued matrix we obtain a right exact sequence
 \begin{equation}
 \Z^k \stackrel{A_W^T}{\longrightarrow} \Z^{n} \longrightarrow \op{coker}(A^T_W) \rightarrow 0
 \end{equation}
Augmenting $A_W^T$ by a row of $-1$s along the bottom, we get another right exact sequence
$$
\Z^k \stackrel{B^T_W}{\longrightarrow} \Z^{n+1} \longrightarrow \op{coker}(B^T_W) \rightarrow 0.
$$

Now apply $\op{Hom}(-, \gm)$ to the above to obtain a left exact sequence
$$
1 \longrightarrow \ker \widehat{B^T_W} \longrightarrow \gm^{n+1} \stackrel{\widehat{B^T_W}}{\longrightarrow} \gm^k.
$$
Note that 
\[
\widehat{B^T_W}(t_1,\ldots, t_{n+1})_i = (t_{n+1}^{-1}\prod_{j=1}^n t_j^{a_{ij}}).
\]
It follows directly from the definition that 
\[
\ker \widehat{B^T_W} = \Gamma_W
\]
where $\Gamma_W$ is defined as in Equation~\eqref{eq: GammaW}. 
Furthermore, when $A_W$ has full rank all the sequences above are exact.

By composing the inclusion $\Gamma_W \rightarrow \gm^{n+1}$  with the projection to the $i$th factor, we obtain characters $\chi_i: \Gamma_W \rightarrow \gm$ for each $i$.  Take $W_i$ to be the restriction of $W$ to the locus where $x_i=1$. 
Then, it is also easy to check that the following sequence is left exact
\begin{equation}
1 \longrightarrow \Gamma_{W_i}  \stackrel{f}{\longrightarrow}  \Gamma_W  \stackrel{\chi_i}{\longrightarrow} \mathbb G_m \label{eq: specializeGamma} 
\end{equation}
where  $f(t_1,\ldots, t_n) = (t_1, \ldots, t_{i-1}, 1, t_{i+1}, \ldots, t_{n})$.

\begin{remark} \label{rem: right exact}
If there exists weights $s_1, ..., s_n$ making $W$ homogeneous and $s_i \neq 0$ then the above sequence is also right exact.  The examples we have in mind are \eqref{eq: Wloopchain} and \eqref{eq: Wchainchain}.  These examples have $n+1$ variables and $W_{n}, W_{n+1}$ are quasi-homogeneous with positive weights.  Hence, the above sequence is exact for all $i$.
\end{remark}

\begin{lemma}  \label{lem: stack iso}
Assume there exists weights $s_j$ making $W$ homogeneous with $s_i \neq 0$.  Then, the inclusion induces an isomorphism of stacks
\[
[\A^n \backslash Z(x_i) / \Gamma_W] \cong  [\A^{n-1} / \Gamma_{W_i} ]
\]
so that $W$ corresponds to $W_i$.
\end{lemma}
\begin{proof}
This follows immediately from \eqref{eq: specializeGamma}, Remark~\ref{rem: right exact},  and Lemma 4.22 of \cite{FK18}.
\end{proof}

\subsection{Exceptional Collections and Tilting}\label{subsec: exceptional}

Let $\mathcal{T}$ be a $\C$-linear triangulated category whose morphism spaces are finite-dimensional $\C$-vector spaces.   For an object $E \in \mathcal{T}$ and $l \in \Z$, we write $E[l]$ as the $l$-fold shift of $E$. 

\begin{definition} Consider $\mathcal{T}$ as above.
\begin{enumerate}[label=(\alph*)]
\item An object $E \in \mathcal{T}$ is called \emph{exceptional} if 
\[
\oplus_{l \in \Z} \Hom_{\mathcal{T}}(E, E[l]) = \Hom_{\mathcal{T}}(E, E) = \C \cdot \text{id}_{E}.
\]
\item A sequence $( E_{1}, \dots, E_{n} )$ of exceptional objects is called an \emph{exceptional collection} if 
\[
\Hom_{\mathcal{T}}(E_{j}, E_{i}[l]) = 0
\]
for all $l \in \Z$ and all $1 \leq i<j \leq n$.
\item An exceptional collection $( E_{1}, \dots, E_{n} )$ is called \emph{strong} if 
\[
\Hom_{\mathcal{T}}(E_{i}, E_{j}[l]) = 0
\]
for all $l \neq 0$ and all $i, j$.
\item An exceptional collection $( E_{1}, \dots, E_{n} )$ is called \emph{full} if it generates $\mathcal{T}$ (i.e. $\mathcal{T}$ is the smallest thick triangulated category containing the objects $E_{1}, \dots, E_{n}$).
\end{enumerate}
\end{definition}

\begin{definition}
An object $T \in \mathcal{T}$ is called \emph{tilting} if 
\begin{enumerate}[label=(\alph*)]
\item $T$ generates $\mathcal{T}$,
\item $\Hom_{\mathcal{T}}(T, T[l]) =0$ for $\l \neq 0$, and
\item the endomorphism algebra $\Hom_{\mathcal{T}}(T, T)$ has finite global dimension.
\end{enumerate}
\end{definition}

The following observation is standard.  We include a brief proof for the reader's convenience.
\begin{proposition}
If $(E_{1}, \dots, E_{n})$ is a full strong exceptional collection then $T = \bigoplus_{i} E_{i}$ is a tilting object. 
\end{proposition}

\begin{proof} First, $T$ generates $\mathcal{T}$ since $E_{1}, \dots, E_{n}$ generate $\mathcal{T}$. Second, by strongness, for $l \neq 0$,
\[
\Hom_{\mathcal{T}}(T, T[l]) = \Hom_{\mathcal{T}}(\oplus_{i} E_{i} , \oplus_{j} E_{j}[l]) =  \oplus_{i, j} \Hom_{\mathcal{T}}(E_{i}, E_{j}[l]) = 0.
\]
Third, the exceptionality of the collection ensures the global dimension of $\Hom_{\mathcal{T}}(T, T)$ is at most $n$, and hence finite.  
\end{proof}

\subsection{Factorization categories}

Let $G$ be an affine algebraic group acting on a smooth variety $X$ over $\C$.  Take $W$ to be a $G$-invariant section of an invertible $G$-equivariant sheaf $\mathcal{L}$, i.e., $W \in \Gamma(X, \mathcal{L})^G$. We call the data $(X, G, W)$ a \newterm{(gauged) Landau-Ginzburg model} and associate the \newterm{absolute derived category} $\dabs{[X, G, W]}$ to this. We refer the reader to \cite{Pos11, BFK14a, BFK14b, EP15, FK18} for background.

We recall the following result of Orlov \cite[Proposition 1.14]{Orl04} in the $G$-equivariant factorization setting.

\begin{proposition} \label{Orlov immersion}
Assume that $[X/G]$ has enough locally free sheaves. Let $i: U\hookrightarrow X$ be a $G$-equivariant open immersion so that the singular locus of $W$ is contained in $U$. Then the restriction
\[
i^*:\dabs{[X, G, W]}\rightarrow \dabs{[U, G, W]}
\]
is an equivalence of categories.
\end{proposition}
\begin{proof}
Consider a matrix factorization $\mathcal E$ with locally-free components $\mathcal E_0, \mathcal E_1$ and maps $\alpha: \mathcal E_0 \to \mathcal E_1, \beta: \mathcal E_1 \to \mathcal E_0 \otimes \mathcal L$ such that $\alpha \circ \beta = \beta \circ \alpha = W$.  Then by the Leibniz rule (i.e. the universal property of K\"ahler differentials), 
\[
dW = d\alpha \circ \beta + \alpha \circ d\beta
\]
i.e. the maps $d\alpha, d\beta$ define a homotopy between the $G$-equivariant morphism of factorizations $dW: \mathcal E \to \mathcal E \otimes \Omega_X$ and $0$.  That is, $\mathcal E$ is annihilated by $dW$.  In summary, since $[X/G]$ has enough locally free sheaves, any factorization is supported on the critical locus of $W$.  

Now for any $\mathcal E$, consider the unit of the adjunction
\[
\mathcal E \to i_*i^*\mathcal E.
\]
The cone of this morphism is, on the one hand, supported on the complement of $U$.  On the other hand, it is supported on the critical locus.  As these do not intersect, the cone has no support.  It follows that the cone is acyclic, or equivalently, the unit of the adjunction is a natural isomorphism.
Conversely, for an open immersion, the counit  $i^* \circ i_* \to \text{Id}$  is always a natural isomorphism.  
\end{proof}

For convenience, we now rewrite Proposition~\ref{Orlov immersion} in our simple algebraic setting. Namely, if $U = \A^n \setminus Z(\mathcal J)\subset X = \A^n \setminus Z(\mathcal I)$, then the containment of the  singular locus $W|_X$  in $U$ is equivalent to the containment of ideals $\mathcal{I} \subseteq \sqrt{\partial W, \mathcal J}$.

\begin{corollary}\label{partial compact toric}
Let $\I$ and $\J$ be two nonzero ideals in $\C[x_1, \ldots, x_n]$ so that $\J \subset \I$. Take $X = \A^n \setminus Z(\I)$ and $U = \A^n \setminus Z(\J)$. Suppose $G$ is a linearly reductive group, the immersion $i:U\hookrightarrow X$ is $G$-equivariant, and $W$ is a $G$-invariant function on $X$. If $\I \subseteq \sqrt{\partial W, \J}$, then 
\[ i^{*}: \dabs{[X, G, W]}\rightarrow \dabs{[U, G, W]}\]
is an equivalence of categories. 
\end{corollary}

\begin{lemma} \label{lem: exc coll}
Let $G$ be an abelian linearly reductive algebraic group lying in an exact sequence
\[
1 \to H \to G \xrightarrow{\chi} \gm \to 1.
\]
Let $S \subseteq \op{Hom}(G, \gm)$ be a set of representatives of the cosets of $\op{Hom}(H, \gm)$.
Then the matrix factorizations 
\[
\{ 0 \substack{\rightarrow\\[-1em] \leftarrow} \C(s) \ | \ s \in S \}
\]
form a full orthogonal (possibly infinite) exceptional collection for $\dabs[\op{Spec}(\C), G, 0]$ where $0$ is a section of $\O(\chi)$.
\end{lemma}
\begin{proof}
We compute
\[
\op{Hom}(0 \substack{\rightarrow\\[-1em] \leftarrow} \C(s_1), 0 \substack{\rightarrow\\[-1em] \leftarrow} \C(s_2)[i] )
\]
for all $i$.
As these matrix factorizations have projective components, we only need to compute homotopy classes of maps between them.  If $i$ is odd, there are no maps.  If $i = 2j$, 
\begin{align*}
\op{Hom}(0 \substack{\rightarrow\\[-1em] \leftarrow} \C(s_1), 0 \substack{\rightarrow\\[-1em] \leftarrow} \C(s_2)[2j] )& = \op{Hom}( \C(s_1), \C(s_2 +\chi^j) ) & \\
& = \begin{cases} 0 & \text{ if }s_1 \neq s_2 + \chi^j \\
\C & \text{ if }s_1 = s_2 \text{ and } j=0 \end{cases} &\text{ by Schur's Lemma.} 
\end{align*}

To see that this set of objects generates $\dabs[\op{Spec}(\C), G, 0]$, notice that $[2] = - \otimes \O(\chi)$.  Hence, they generate all objects of the form $0 \substack{\rightarrow\\[-1em] \leftarrow} \C(\tau)$ with $ \tau \in \text{Hom}(G, \gm)$.  Since $G$ is abelian, this is all irreducible representations of $G$.
It is easy to see that this new set generates.  Indeed by Schur's Lemma again, all objects are sums of shifts of these objects.
\end{proof}

\subsection{Milnor Numbers}
\begin{definition}
Suppose $w \in \C[x_1, ..., x_n]$  has an isolated singularity.  We define the \textbf{Milnor number} of $w$ by the formula
\[
\mu(w) := \text{dim } \C[x_1, ..., x_n] / \langle \partial_{x_1} w, ..., \partial_{x_n} w \rangle.
\]
\end{definition}

The following lemmas provide a formula for the Milnor number of any invertible polynomial.

\begin{lemma} \label{lem: milnor multiplicative}
Suppose $w \in \C[x_1, ..., x_n]$ and $v \in \C[y_1, ..., y_m]$ have isolated singularities.  Then
\[
\mu(w+v) = \mu(w) \mu(v).
\] 
\end{lemma}
\begin{proof}
We have
\begin{align*}
\mu(w + v)  & = \text{dim } \C[x_1, ..., x_n, y_1, ..., y_m] / \langle \partial_{x_1} w, ..., \partial_{x_n} w, \partial_{y_1} v, ..., \partial_{y_m} v  \rangle \\
& = \text{dim } \C[x_1, ..., x_n] / \langle \partial_{x_1} w, ..., \partial_{x_n} w \rangle \otimes  \C[y_1, ..., y_m] / \langle \partial_{y_1} v, ..., \partial_{y_m} v  \rangle \\
& = \text{dim } \C[x_1, ..., x_n] / \langle \partial_{x_1} w, ..., \partial_{x_n} w \rangle \text{ dim }  \C[y_1, ..., y_m] / \langle \partial_{y_1} v, ..., \partial_{y_m} v  \rangle \\
& = \mu(w) \mu(v).
\end{align*}
\end{proof}

\begin{theorem}[Milnor-Orlik] \label{thm: milnor num}
If  $w = x_1^{t_1}x_2 + ... + x_n^{t_n}x_1$ is a loop polynomial,
then
\[
\mu(w^T) = \prod_{i=1}^{n} t_i.
\]
If $w = x_1^{t_1}x_2 + ... + x_n^{t_n}$ is a chain polynomial, then
\[
\mu(w^T) = \sum_{k=0}^{n} (-1)^{n-k} \prod_{i=1}^{k} t_{i}.
\]
where the empty product is one, by convention. 
\end{theorem}
\begin{proof}
These formulas can be obtained by plugging the appropriate weights into \cite[Theorem 1]{MO70}.  It can also be obtained from \cite[Theorem 2.10]{HLSW} where they give explicit bases.
\end{proof}

\begin{remark}
As the Milnor number $\mu(w^T)$ is the dimension of the state space of the mirror Landau-Ginzburg model $(\A^n, w^T)$, we expect that, in connection with Conjecture 1.4 of \cite{HO}, the full exceptional collection of the category $\dabs[\A^n, \Gamma_w, w]$ will have length $\mu(w^T)$. We show this in the next section.
\end{remark} 

\subsection{Elementary Geometric Invariant Theory}
Fix an algebraic group $\Gamma$ and a group homomorphism $\Gamma \to \gm^n \subseteq GL_n$ which gives rise to a diagonal action of $\Gamma$ on $\mathbb A^n$. A choice of one-parameter subgroup $\lambda: \gm \to \Gamma$ can be described by a sequence of weights $c_1, ..., c_n$.  We can then define ideals
\begin{align*}
\mathcal I_+ & := \langle x_i \ | \ c_i > 0 \rangle \\
 \mathcal I_- & := \langle x_i \ | \ c_i < 0 \rangle.
\end{align*}
This gives rise to two global quotient stacks which we call the positive and negative \emph{$(\Gamma,\lambda)$-geometric invariant theory (GIT) quotients} respectively
\begin{align*}
X_\pm := [\mathbb A^n \backslash Z(\mathcal I_\pm) / \Gamma].
\end{align*}
\begin{remark}
Notice that in the definition above, the semi-stable loci are obtained strictly from the $\gm$-action induced by $\lambda$.  However, the quotients are by $\Gamma$ as opposed to this $\gm$.  
\end{remark}

\section{Existence of Exceptional Collections}\label{VGIT KS Section}

\subsection{Warm-up: Exceptional Collections for Fermat Polynomials}\label{warmup fermat}

For the sake of completeness, we will show that that $\dabs{[\A^1, \Gamma_w, w]}$ has an exceptional collection for $w=x_1^r$.  This result is well known, quite simple by hand, and is also a consequence of a theorem of Orlov \cite[Corollary 2.9]{Orl09}.  The difference in our approach is that we will use VGIT to obtain the result.  We do this to illustrate that our entire article is a consequence of VGIT for categories of factorizations \cite{BFK12} and the Thom-Sebastiani formula for gauged LG models \cite{BFK14a, BFK14b}.

Consider the polynomial $W = x_2x_1^r$ and define $w_+ := W_2 = x_1^r$ and $w_- := W_1 = x_2$.
  Let $c_2 = r$ and $c_1 =  -1$. The $c_i$ determine a diagonal one-parameter subgroup of $\Gamma_W$ by the map $\lambda: \gm \to \Gamma_W$ under the map $\gamma(t) = (t^{c_2}, t^{c_1}, 1)$. The semistable loci for this one parameter subgroup are
\[
U_+ := \A^2 \setminus Z(x_2); \quad U_- :=\A^2 \setminus Z(x_1).
\]
By Lemma~\ref{lem: stack iso}, we see that $[U_\pm/\Gamma_W] = [\mathbb{A}^1/ \Gamma_{w_\pm}]$. Notice that 
\begin{align*}
\dabs{[\op{Spec}(\C), \Gamma_W/\lambda(\gm), 0]} & \cong \dbcoh{[\op{Spec}(\C)]} & \text{ by  \cite[Corollary 2.3.12]{BFK12}} \\
& \cong \langle E \rangle & \text{ where }E \text{ is the exceptional object }\C\\
\end{align*}
Hence,  
\begin{align*}
\dabs{[\A^1, \Gamma_{x_1^r}, x_1^r]} & \cong \langle E_1, \ldots, E_{r-1}, \dabs{[\A^1, \Gamma_{x_2}, x_2]}\rangle  & \text{by \cite[Theorem 3.5.2 (a)]{BFK12} }\\
& \cong \langle E_1, \ldots, E_{r-1} \rangle  & \text{since }x_2\text{ has no critical locus}\\
\end{align*}

\subsection{Exceptional Collections for Loop Polynomials}\label{section loop to chain}
For any natural numbers $a_i, b \geq 2$, consider the polynomial
\begin{equation} \label{eq: Wloopchain}
W := x_1^{a_1} x_2 + x_2^{a_2}x_3 + \ldots + x_{n-1}^{a_{n-1}}x_n + x_n^{a_n}x_1x_{n+1}^{b}.
\end{equation}
Then, 
\begin{equation} \label{eq: loop+}
w_+ := W_{n+1} = x_1^{a_1} x_2 + x_2^{a_2}x_3 + \ldots + x_{n-1}^{a_{n-1}}x_n + x_n^{a_n}x_1
\end{equation}
is a loop polynomial
and
\begin{equation} \label{eq: loop-}
w_- := W_n = x_1^{a_1} x_2 + x_2^{a_2}x_3 + \ldots + x_{n-1}^{a_{n-1}} + x_1x_{n+1}^{b}
\end{equation}
is a chain polynomial.  In this section we will show that the derived categories of the gauged Landau-Ginzburg models associated to $w_+, w_-$ differ by an exceptional collection.

Let $(-1)^{i+n+1} d_i$ be the determinant of the $i^{th}$ maximal minor of the matrix $A_W$ and
\[
c_i := \frac{d_i}{gcd(d_{1}, ..., d_{n+1})}.
\]
  Explicitly in this case, 
\begin{equation}\begin{aligned}\label{pi equation for loop}
d_1 &= (-1)^{n} b; \\
d_j &= (-1)^{j+n+1} b\prod_{i=1}^{j-1} a_{i} \text{ for $2 \leq j\leq n$}; \text{ and} \\
d_{n+1} &=  a_1 \cdots a_n + (-1)^{n+1}.
\end{aligned}\end{equation}
It is easy to check that the $c_i$ determine a diagonal one-parameter subgroup
\begin{align*}
\lambda: \gm & \to \Gamma_W \\
t & \mapsto (t^{c_1}, ..., t^{c_{n+1}}, 1).
\end{align*}

We define
\[
U_+ := \A^{n+1} \backslash Z(x_{n+1}), U_- := \A^{n+1} \backslash Z(x_n).
\]

\begin{remark} \label{rem: ci}
The $c_i$ are the unique (up to sign) relatively prime weights of the $x_i$ such that $W$ is homogeneous of degree zero.  We fix our sign convention so that $c_{n+1}$ is positive and $c_n$ is negative.  This ensures that $\A^{n+1} \backslash Z(\mathcal I_\pm) \supseteq U_\pm$.
\end{remark}

\begin{lemma}
\label{lem: critical localize}
There are equivalences of categories
\[
\dabs[X_\pm, W] \cong \dabs[U_\pm, \Gamma_W, W].
\]

\end{lemma}
\begin{proof}

Since $Z(x_{n+1}), Z(x_n)$ are $\Gamma_W$ invariant, the open immersions
\[
i_\pm: U_\pm \hookrightarrow \A^{n+1} \backslash Z(\mathcal I_\pm)
\]
are $\Gamma_W$-equivariant.
Hence, by Corollary~\ref{partial compact toric}, the statement of the lemma reduces to proving the containments 
 \[
 \I_+ \subseteq \sqrt{\partial W, x_{n+1}} \ \  \text{ and } \ \ \I_- \subseteq \sqrt{\partial W, x_n}.
 \]  
 From the partial derivative $\partial_{x_n}W = x_{n-1}^{a_{n-1}} + a_nx_1x_n^{a_n-1}x_{n+1}^{b}$, we see that $x_{n-1} \in \sqrt{\partial W, x_{n+1}}$  (respectively $\sqrt{\partial W, x_n}$). For $1 < i < n$, we compute $\partial_{x_{i}}W = x_{i-1}^{a_{i-1}} + a_i x_i^{a_i-1}x_{i+1}$. Hence, if $x_i \in \sqrt{\partial W, x_{n+1}}$ (respectively $\sqrt{\partial W, x_n} $) then $x_{i-1} \in \sqrt{\partial W, x_{n+1}}$ (respectively $\sqrt{\partial W, x_n}$). Both containments follow from descending induction. 
\end{proof}

\begin{lemma} \label{lem: milnor number}
The following identity holds.
\[
\mu(w_+^T) - \mu(w_-^T) = \sum d_i
\]
\end{lemma}
\begin{proof}
This is a simple calculation plugging in the Milnor numbers carefully from Theorem~\ref{thm: milnor num}.
\end{proof}

\begin{theorem}\label{compare chain and loop}
Take the polynomials 
\[
w_+ = W_{n+1} = x_1^{a_1} x_2 + x_2^{a_2}x_3 + \ldots + x_{n-1}^{a_{n-1}}x_n + x_n^{a_n}x_1
\]
and
\[
w_- = W_n = x_1^{a_1} x_2 + x_2^{a_2}x_3 + \ldots + x_{n-1}^{a_{n-1}} + x_1x_{n+1}^{b}
\]
for $a_i \geq 2$ and $b\geq 2$.

The following statements hold:
\begin{enumerate}[label=(\alph*)]

\item If $\mu(w_-^T) > \mu(w_+^T)$, then we have a semi-orthogonal decomposition
\[
\dabs[ \A^n, \Gamma_{w_-}, w_-] \cong \langle E_1, \ldots, E_{\mu(w_-^T) - \mu(w_+^T)}, \dabs[ \A^n, \Gamma_{w_+}, w_+]\rangle
\]
where each  $E_j$ is an exceptional object.
\item If $\mu(w_+^T) = \mu(w_-^T)$, then we have the equivalence
\[
\dabs[ \A^n, \Gamma_{w_+}, w_+] \cong \dabs[ \A^n, \Gamma_{w_-}, w_-].
\] 
\item If $\mu(w_+^T) > \mu(w_-^T)$, then we have a semi-orthogonal decomposition
\[
\dabs[ \A^n, \Gamma_{w_+}, w_+] \cong \langle E_1, \ldots, E_{\mu(w_+^T) - \mu(w_-^T)}, \dabs[ \A^n, \Gamma_{w_-}, w_-]\rangle
\]
where each  $E_j$ is an exceptional object.
\end{enumerate}
\end{theorem}

\begin{proof}
We have a sequence of equivalences using Lemmas~\ref{lem: stack iso} and~\ref{lem: critical localize}:
\[\dabs[ \A^n, \Gamma_{w_+}, w_+]   \cong \dabs[ U_+, \Gamma_W, W]  \cong \dabs[ X_+, W];\]
\[\dabs[ \A^n, \Gamma_{w_-}, w_-]   \cong \dabs[ U_-, \Gamma_W, W]  \cong \dabs[ X_-, W].\]
We then apply \cite[Theorem 3.5.2]{BFK12} to get
\begin{enumerate}[label=(\alph*)]
\item If $\sum_i c_i < 0$, then we have a semi-orthogonal decomposition
\[
\dabs[ X_-, W] \cong \langle E_1, \ldots, E_t, \dabs[X_+, W]\rangle,\]
\item If $\sum_i c_i = 0$ then we have the equivalence 
\[\dabs[ X_-, W] \cong \dabs[X_+, W], \text{ and }\]
\item If $\sum_i c_i >0$, then we have a semi-orthogonal decomposition
\[
\dabs[ X_+, W] \cong \langle E_1, \ldots, E_t, \dabs[X_-, W]\rangle,\]
\end{enumerate}
where each $E_j$ is an exceptional object (explained below).  These correspond to the cases of the theorem by Lemma~\ref{lem: milnor number}.

To clarify the appearance of exceptional objects, notice that all the $c_i$ are non-zero.  Hence, the fixed locus of $\lambda$ is just the origin.  Let $\overline{\chi_{n+1}}$ be the character of $\Gamma_W / \lambda$ induced by $\chi_{n+1}$.  By \cite[Remark 4.2.3]{BFK12} the orthogonal components are all equivalent to $\dabs[\op{Spec}(\C), \Gamma_W /\lambda, 0]$ where $0$ is a section of $\O(\overline{\chi_{n+1}})$.  This category has an exceptional collection by Lemma~\ref{lem: exc coll} of length $|\op{ker} \overline{\chi_{n+1}}|$.

Now, let us calculate $t$.  In the statement of \cite[Theorem 3.5.2]{BFK12}, the category $\dabs[\op{Spec}(\C), \Gamma_W /\lambda, 0]$ occurs $|\sum c_i|$ times.  Hence $t = |\op{ker} \overline{\chi_{n+1}}| |\sum {c_i}|$.
By the snake lemma, $\text{Hom}(\op{ker} \overline{\chi_{n+1}}, \gm)$ is isomorphic to the torsion subgroup of the cokernel of $A^T_W$.
Since the $d_i$ are the determinants of the maximal minors of this matrix, $|\op{ker} \overline{\chi_{n+1}}| = \text{gcd}(d_1, ..., d_{n+1})$.  
   Hence, $t = |\op{ker} \overline{\chi_{n+1}}| |\sum {c_i}| = |\sum d_i|$ which equals $|\mu(w_+^T) - \mu(w_-^T)|$ by Lemma~\ref{lem: milnor number}.
\end{proof}

We now compute the difference of the Milnor numbers to apply Theorem~\ref{compare chain and loop}.

\begin{lemma}\label{sign of sum of cis}
If $b \leq a_n$, then $\mu(w_+^T) - \mu(w_-^T) > 0$. 
\end{lemma}

\begin{proof}
By Lemma~\ref{lem: milnor number}, it is equivalent to prove that the sum of the $d_i$ is positive. If $n$ is even, then, since $a_k \geq 2$ for all $k$, we have

\begin{equation}\begin{aligned}
\sum_{i=1}^{n+1} d_i &=(a_1\cdots a_n-1) + b+\left(\sum_{j=1}^{n} (-1)^{j+1} b\prod_{i=1}^{j-1} a_{i}\right) - a_1\cdots a_{n-1}b\\ 
	&\geq (b - 1) +\left(\sum_{j=1}^{n} (-1)^{j+1} b\prod_{i=1}^{j-1} a_{i}\right) \\
	&= (b-1) +  \sum_{k=1}^{n/2 -1} (a_{2k}-1) a_{1}\cdots a_{2k-1}b \\
	&> 0.
\end{aligned}\end{equation}

If $n$ is odd, then we have 

\begin{equation}\begin{aligned}
\sum_{i=1}^{n+1} d_i &=(a_1\cdots a_n+1) + -b + \left(\sum_{j=1}^{n} (-1)^{j} b\prod_{i=1}^{j-1} a_{i}\right) - a_1\cdots a_{n-1}b\\ 
	&\geq 1-b+\left(\sum_{j=1}^{n} (-1)^{j} b\prod_{i=1}^{j-1} a_{i}\right) \\
	&=  1+ b(a_1-1) + \sum_{k=2}^{(n-1)/2} (a_{2k-1}-1) a_{1}\cdots a_{2k-2}b \\
	&> 0.
\end{aligned}\end{equation}
\end{proof}

\begin{corollary} \label{cor: comparison}
If  $b \leq a_n$ and $\dabs[ \A^n , \Gamma_{w_-}, w_-]$ has a full exceptional collection of length $\mu(w_-^T)$, then $\dabs[ \A^n, \Gamma_{w_+}, w_+]$ has a full exceptional collection of length $\mu(w_+^T)$.
\end{corollary}
\begin{proof}
By Lemmas~\ref{lem: milnor number} and \ref{sign of sum of cis}, we can apply Theorem~\ref{compare chain and loop}(c). 
The result follows immediately.
\end{proof}

\subsection{Exceptional Collections for Chain Polynomials} \label{sec: chain to smaller chain}
In this subsection, we argue that the derived category of a chain polynomial admits a full exceptional collection.  We omit most of the details as the proof is nearly identical to the one appearing in the previous section.  Moreover, this result already appeared recently \cite[Corollary 1.6]{HO}.  Nevertheless, we provide the reader with the appropriate changes for a self-contained treatment of the entire result using just VGIT and the Thom-Sebastiani formula for gauged LG models.

For any $b\geq 2$, consider the polynomial
\begin{equation} \label{eq: Wchainchain}
W := x_1^{a_1} x_2 + x_2^{a_2}x_3 + \ldots + x_{n-1}^{a_{n-1}}x_n + x_n^{a_n}x_{n+1}^{b}.
\end{equation}
Then
\begin{equation} \label{eq: chain+}
w_+ := W_{n+1} = x_1^{a_1} x_2 + x_2^{a_2}x_3 + \ldots + x_{n-1}^{a_{n-1}}x_n + x_n^{a_n}
\end{equation}
is a chain polynomial of length $n$
and
\begin{equation}  \label{eq: chain-}
w_- := W_{n} = x_1^{a_1} x_2 + x_2^{a_2}x_3 + \ldots + x_{n-1}^{a_{n-1}} + x_{n+1}^{b}.
\end{equation}
is a Thom-Sebastiani sum of a chain polynomial of length $n-1$ and a Fermat polynomial. 

Again, we consider the diagonal one-parameter subgroup of $\Gamma_W$ defined as the image of the map
\begin{align*}
\lambda: \gm & \to \Gamma_W \\
t & \mapsto (t^{c_1}, ..., t^{c_{n+1}}, 1)
\end{align*}
where, again,  the $(-1)^{i+n + 1}c_i$ are the determinants of the full rank minors of $A_W$ divided by their greatest common divisor.  Explicitly,
\begin{equation}\begin{aligned}\label{pi equation for chain}
d_j &= (-1)^{n+j-1} b\prod_{i=1}^{j-1} a_i, \text{ for $1 \leq j\leq n$}, \\
d_{n+1} &=  a_1\cdots a_n, \\
c_j & = \frac{d_j}{gcd(d_1, ..., d_{n+1})}.
\end{aligned}\end{equation}

We define
\begin{align*}
U_+ &:= \A^{n+1} \backslash Z(x_{n+1}) \ \ \text{ and } \ \ U_- := \A^{n+1} \backslash Z(x_n), \\
\mathcal I_+  &= \langle x_{n+1}, x_j \ | \ j \not \equiv n \pmod 2\rangle, \text{ and}\\
\mathcal I_- &= \langle x_j \ | \ j \equiv n \pmod 2\rangle.
\end{align*}

\begin{lemma}\label{lem: critical localize chain}
There are equivalences of categories
\[
\dabs[X_\pm, W] \cong \dabs[U_\pm, \Gamma_W, W].
\]
\end{lemma}

\begin{proof}
The proof is almost the same as that of Lemma~\ref{lem: critical localize}. The only difference is the computation of $\partial_{x_n}W$; however, the conclusion  that  $x_{n-1} \in \sqrt{\partial W, x_{n+1}}$ (respectively $\sqrt{\partial W, x_n}$) still holds.
\end{proof}

\begin{lemma} \label{lem: milnor number2}
The following identity holds.
\[
\mu(w_+^T) - \mu(w_-^T) = \sum d_i 
\]
\end{lemma}
\begin{proof}
Again, this is a simple calculation using Lemma~\ref{lem: milnor multiplicative} and Theorem~\ref{thm: milnor num}.
\end{proof}

\begin{theorem}\label{compare chain and smaller chain}
Take the polynomials 
\[
w_+ := W_{n+1} = x_1^{a_1} x_2 + x_2^{a_2}x_3 + \ldots + x_{n-1}^{a_{n-1}}x_n + x_n^{a_n}
\]
and
\[
w_- := W_{n} = x_1^{a_1} x_2 + x_2^{a_2}x_3 + \ldots + x_{n-1}^{a_{n-1}} + x_{n+1}^{b}
\]
for $a_i \geq 2$ and $b\geq 2$. 
The following statements hold:
\begin{enumerate}[label=(\alph*)]

\item If $\mu(w_+^T) < \mu(w_-^T)$, then we have a semi-orthogonal decomposition
\[
\dabs[ \A^n, \Gamma_{w_-}, w_-] \cong \langle E_1, \ldots, E_{\mu(w_-^T) - \mu(w_+^T)}, \dabs[ \A^n, \Gamma_{w_+}, w_+]\rangle
\]
where each  $E_j$ is an exceptional object.
\item If $\mu(w_+^T) = \mu(w_-^T)$, then we have the equivalence
\[
\dabs[ \A^n, \Gamma_{w_+}, w_+] \cong \dabs[ \A^n, \Gamma_{w_-}, w_-].
\] 
\item If $\mu(w_+^T) > \mu(w_-^T)$, then we have a semi-orthogonal decomposition
\[
\dabs[ \A^n, \Gamma_{w_+}, w_+] \cong \langle E_1, \ldots, E_{\mu(w_+^T) - \mu(w_-^T)}, \dabs[ \A^n, \Gamma_{w_-}, w_-]\rangle
\]
where each  $E_j$ is an exceptional object.
\end{enumerate}
\end{theorem}

\begin{proof}
The proof is verbatim as in Theorem~\ref{compare chain and loop} using Lemma~\ref{lem: critical localize chain}  instead of Lemma~\ref{lem: critical localize} and Lemma~\ref{lem: milnor number2} instead of Lemma~\ref{lem: milnor number}.
\end{proof}

Again, we compute the sign of difference of the Milnor numbers to apply the theorem.

\begin{lemma}\label{sign of sum of cis chain}
If $b \leq a_n$, then $\mu(w_+^T) - \mu(w_-^T) \geq 0$. 
\end{lemma}

\begin{proof}
By Lemma~\ref{lem: milnor number2}, it is equivalent to show that $\sum_i d_i \geq 0$. If $n$ is odd, then we have that
$$
\sum_{i=1}^{n+1} d_i= (a_n-b) a_1 \cdots a_{n-1} + \sum_{k=1}^{(n-1)/2} (a_{2k-1} -1) a_1 \cdots a_{2k-2}b \geq 0.
$$
If $n$ is even, then we have that
$$
\sum_{i=1}^{n+1} d_i= (a_n-b) a_1 \cdots a_{n-1} + \left(\sum_{k=1}^{(n-2)/2} (a_{2k} -1) a_1 \cdots a_{2k-1}\right)b + b > 0.
$$
\end{proof}

We now reprove Corollary 1.6 of \cite{HO}.

\begin{corollary}\label{Exceptional Collection Chain}
Let $w_{chain} = x_1^{a_1} x_2 + x_2^{a_2}x_3 + \ldots + x_{n-1}^{a_{n-1}}x_n + x_n^{a_n}$ be a chain polynomial of length $n$ with $a_i \geq 2$. Then $\dabs{[ \A^n, \Gamma_{w_{chain}}, w_{chain}]}$ has a full exceptional collection of length $\mu(w_{chain})$.
\end{corollary}

\begin{proof}
We proceed by induction on $n$.  The base case $n=1$ is contained in \S\ref{warmup fermat}.

Now let $n>1$ and choose $b \leq a_n$. Consider the polynomials $W$, $w_+$ and $w_-$ as above. The polynomial $w_-$ is the Thom-Sebastiani sum of two polynomials $x_{n+1}^b$ and $x_1^{a_1}+\ldots x_{n-2}^{a_{n-2}}x_{n-1} + x_{n-1}^{a_{n-1}}$, hence, by the induction hypothesis, Lemma~\ref{lem: milnor multiplicative}, and Corollary 2.40 of \cite{BFK14b}, the derived category $\dabs{ [\A^n, \Gamma_{w_{-}}, w_{-}]}$ has an exceptional collection of length $\mu(w_-^T)$. By Lemmas~\ref{lem: milnor number2} and  \ref{sign of sum of cis chain}, the inequality $\mu(w_+^T) \geq \mu(w_-^T)$ holds. Apply case (b) or (c) of Theorem~\ref{compare chain and smaller chain} to see that 
$$
\dabs[ \A^n, \Gamma_{w_+}, w_+] \cong \langle E_1, \ldots, E_{\mu(w_+^T)  - \mu(w_-^T) }, \dabs[ \A^n, \Gamma_{w_-}, w_-]\rangle,
$$
hence $\dabs[ \A^n, \Gamma_{w_+}, w_+]$ has a semi-orthogonal decomposition of objects which have an exceptional collection, hence it has an exceptional collection.
\end{proof}

\subsection{The Gorenstein Case}

\begin{definition}
Let $w, v$ be invertible polynomials.  We say that $w,v$ are related by a \textbf{Kreuzer-Skarke cleave} if they have the same Milnor number and $A_w, A_v$ differ by only one row.
\end{definition}

\begin{corollary} \label{cor: KS cleaves}
Suppose $w,v$ are related by a sequence of Kreuzer-Skarke cleaves.  Then there is an equivalence of categories
\[
\dabs[ \A^n, \Gamma_{w}, w] \cong \dabs[ \A^n, \Gamma_{v}, v].
\]
\end{corollary}
\begin{proof}
This follows immediately from Theorems  \ref{compare chain and loop} and \ref{compare chain and smaller chain}.
\end{proof}

\begin{lemma} \label{lem: KS Goren}
Let $w$ be an invertible polynomial whose dual polynomial $w^T$ quasi-homogeneous has weights $r_i$ and degree $d^T$.  Assume $r_i$ divides $d^T$ for all $i$. Then $w$ is related to  $\sum x_i^{d^T/r_i}$ by a sequence of Kreuzer-Skarke cleaves.
\end{lemma}

\begin{proof}
The proof is the same for the setups in \S\ref{section loop to chain} and \S\ref{sec: chain to smaller chain}, so we prove them simultaneously. 
First, note that $d_{n+1} = \det A_w$ and by Cramer's rule $d_j = -b\det A_w(A_w^{-1})_{jn}$ for $1\leq j \leq n$. Furthermore, the weights $r_i$ of the dual polynomial $w^T$ are obtained by the formula $r_i = \sum_{j=1}^n (A_w^{-1})_{ji} d^T$. 
We see that
\[
\sum_{i=0}^n d_j = \det A_w  \left ( 1- \sum_{j=1}^n b (A_w^{-1})_{jn} \right ) = \det A_w \left ( 1 - \frac{br_n}{d^T} \right ).
\]
If we take $b = d^T / r_n$, we have that $\sum_{i=0}^n d_i = 0$, hence $\sum_{i=0}^n c_i = 0$.  

If we start with a loop, we use the setup in \S\ref{section loop to chain} to obtain a chain. If we have a chain of length $n$, we use the setup in \S\ref{sec: chain to smaller chain} to get the Thom-Sebastiani sum chain of length $n-1$ and a Fermat polynomial. Since $r_i$ divides $d$ for all $i$, we can iterate the process, ending with a Fermat polynomial. 
\end{proof}

\begin{corollary} \label{cor: fermat case}
Let $w$ be an invertible polynomial.  Assume that the dual polynomial $w^T$ has weights $r_i$ such that $r_i$ divides the degree $d^T$.  Then, there is an equivalence of categories.
\[
\dabs[ \A^n, \Gamma_{w}, w] \cong \dabs[ \A^n, \Gamma_{\sum x_i^{d^T / r_i}}, \sum x_i^{d^T / r_i}].
\]
\end{corollary}
\begin{proof}
This follows immediately from Corollary~\ref{cor: KS cleaves} and Lemma~\ref{lem: KS Goren}.
\end{proof}

\subsection{Proof of Theorem~\ref{proof of conjecture}}\label{proof of HO conjecture}

\begin{proof}[Proof of Theorem~\ref{proof of conjecture}]
Recall that the Kreuzer-Skarke classification \cite{KS92} states that an invertible polynomial is the Thom-Sebastiani sum of the following types of polynomials:
\begin{enumerate}
\item Fermat type: $w = x^r$,
\item Chain type: $w = x_1^{a_1}x_2 + x_2^{a_2} x_3 + \ldots + x_{n-1}^{a_{n-1}}x_n + x_n^{a_n}$, and 
\item Loop type: $w = x_1^{a_1}x_2 + x_2^{a_2} x_3 + \ldots + x_{n-1}^{a_{n-1}}x_n + x_n^{a_n}x_1$.
\end{enumerate}

By Lemma~\ref{lem: milnor multiplicative} and Corollary 2.40 of \cite{BFK14b},  the statement of the corollary reduces to proving that $\dabs{[\mathbb{A}^n, \Gamma_w, w]}$ has a full exceptional collection for any of the cases above (without taking a Thom-Sebastiani sum).  The Fermat type case is proven in \cite[Corollary 2.9]{Orl09} or in \S\ref{warmup fermat}. The chain case is proven in \cite[Corollary 1.6]{HO} or Corollary~\ref{Exceptional Collection Chain}.  The loop case is then deduced from applying Corollary~\ref{cor: comparison}.
The special case where we get a tilting object follows from Corollary~\ref{cor: fermat case}.
\end{proof}

\end{document}